\documentclass[a4paper, 12pt, leqno]{amsart}

\usepackage[dvipsnames]{xcolor}

\usepackage{iftex}
\ifPDFTeX % PDFLaTeX
  \usepackage[utf8]{inputenc}
  \usepackage[T1]{fontenc}
\else % LuaLaTeX & XeLaTeX
  \usepackage{fontspec}
\fi
\usepackage{lmodern}

\usepackage{geometry}
\usepackage[textwidth=28mm]{todonotes}
\usepackage{mathtools}
\usepackage{yfonts}
\usepackage[foot]{amsaddr}

\usepackage[pdfencoding=auto]{hyperref}
\hypersetup{
    colorlinks,
    citecolor=green,
    filecolor=black,
    linkcolor=blue,
    urlcolor=black
}

\usepackage[english]{babel}

\usepackage{dsfont}
\usepackage{soul}
\usepackage{comment}
\usepackage{cases}
\usepackage{enumerate}
\usepackage[shortlabels]{enumitem}
\usepackage{tikz}
\usepackage{pgfplots}
\pgfplotsset{compat=1.15}
\usetikzlibrary{arrows}
\usetikzlibrary{patterns}
\usepackage{amsmath,amsfonts,amssymb,amsthm}
\numberwithin{equation}{section}

\usepackage{tabularx}

\newcommand{\nocontentsline}[3]{}
\let\origcontentsline\addcontentsline
\newcommand\stoptoc{\let\addcontentsline\nocontentsline}
\newcommand\resumetoc{\let\addcontentsline\origcontentsline}

\theoremstyle{plain}
\newtheorem{thm}{Theorem}[section]

\newtheorem{lem}[thm]{Lemma} 
\newtheorem{rem}[thm]{Remark} 

\newtheorem{prop}[thm]{Proposition} 
\theoremstyle{definition} 
\newtheorem{defi}{Definition}[section] 
\theoremstyle{remark}

\newcommand{\R}{\mathbb{R}}

\newcommand{\x}{x}

\newcommand{\uv}{v}
\newcommand{\f}{\mathbf{f}}

\newcommand{\eps}{\varepsilon}

\makeatletter
\newcommand{\leqnomode}{\tagsleft@true}
\newcommand{\reqnomode}{\tagsleft@false}
\makeatother

\usepackage[normalem]{ulem}

\definecolor{gemgreen}{RGB}{0,150,0} % Custom green tone for minor typos
\definecolor{revblue}{RGB}{0,80,200} % Correction blue

\author{Gabriella Puppo}
\author{Thomas Rey}
\author{Tommaso Tenna}
\address[Gabriella Puppo]{Dipartimento di Matematica, Sapienza Università di Roma, P.le Aldo Moro 5, 00185 Rome, Italy}
\address[Thomas Rey]{Université Côte d’Azur, CNRS, LJAD, Parc Valrose, F-06108 Nice, France}
\address[Tommaso Tenna]{Université Côte d’Azur, CNRS, LJAD, Parc Valrose, F-06108 Nice, France -- Dipartimento di Matematica, Sapienza Università di Roma, P.le Aldo Moro 5, 00185 Rome, Italy}

\title[]{Hydrodynamic limit of a viscous BGK model of the Boltzmann equation}
\date{}

\begin{document}

\begin{abstract}
    
    In this paper, we investigate a viscous regularization for the BGK model of the Boltzmann equation and its hydrodynamic limit. First, we establish rigorous convergence toward the viscous Euler equations, by exploiting bounds derived from spatial Fisher information. We then analyze the formal derivation of the Euler equations with an entropy inequality for both single- and multi-species case, comparing our results against symmetric hyperbolic thermodynamically compatible theory. 
    \medskip
    
    \textsc{2020 Mathematics Subject Classification:} 82B40, %Kinetic theory of gases
    76P05, %Rarefied gas flows, Boltzmann equation
    35D40. %Viscosity solutions to PDEs
\end{abstract}

\keywords{Hydrodynamic limit; Boltzmann equation; BGK equation; artificial viscosity.}

\maketitle

\vspace{-1cm}
% \tableofcontents

\section{Introduction}

The main goal of this work is providing some insights about the hydrodynamic limit of a viscous BGK model of the Boltzmann equation \cite{BGK1954}. Several aspects motivate the choice of considering a viscous version of the BGK model.\\
The first aspect concerns the role of the spatial diffusion, which offers several technical advantages, providing ``\textit{smoothing}'' properties of the operator. This is strictly connected to the vanishing viscosity approach introduced for hyperbolic conservation laws, see \cite{bianchinibressan2005} and references therein. The regularity introduced by the diffusion term provides control on all macroscopic moments of interest. In classical BGK, the absence of spatial regularity on the distribution function implies that the only compactness estimates for moments come from velocity averaging lemmas, which yield only fractional Sobolev regularity, see \cite{golse1988, golse2005_parma}. Moreover, the diffusion operator acts as regularizing term, giving compactness in the phase space, without considering renormalized solutions. Such properties allow for a rigorous analysis of the limit to the viscous Euler equations, exploiting the uniform estimates obtained through the entropy inequality.\\

The second aspect regards the investigation of the formal hydrodynamic limit towards the viscous Euler equations provided with an entropy inequality. This framework appears in symmetric hyperbolic and thermodynamically compatible (SHTC) theory \cite{ruggeri1981, thomann2023}, where the equation for the energy is an extra conservation law, replaced by the entropy production equation. Due to this choice, energy conservation is ensured by imposing conditions on the entropy production term. In this work, we will show that the viscous Euler equations with the entropy inequality can be obtained as limit of the viscous BGK equation without any further assumption. Indeed, since the viscous BGK equation preserves energy, the entropy production term obtained at the macroscopic level is automatically consistent with such conservation law.

Let $f^\eps: \R^+ \times \R^d \times \R^d \to \R^d$ be a non-negative distribution function, solution to the Cauchy problem 
\begin{equation}
\label{eq:BGK_viscous}
    \begin{cases}
        \partial_t f^\eps + v \cdot \nabla_x f^\eps = \displaystyle \frac{1}{\eps} \left( \mathcal{M}[f^\eps] - f^\eps \right) + \tau \Delta_x f^\eps,\\
        f^\eps(0,x,v) = f^\eps_0(x,v),
    \end{cases}
\end{equation}
for $\tau>0$ a real coefficient. This corresponds to a Boltzmann-type equation with an ``artificial spatial viscosity'' term. From a physical perspective, this Laplacian can be interpreted as an uncertainty in the spatial position of the particles governed by a Brownian distribution.\\
The parameter $\eps >0$ is the \textit{Knudsen number}, defined as the ratio between the mean free path of particles and the physical length scale of observation. This quantity governs the frequency of collisions: for $\eps \ll 1$ the evolution of the particle system is governed by collisions and the gas is \textit{dense}, while, if $\eps \sim 1$ the evolution of the particle system is governed by the advection term and the gas can be considered \textit{rarefied}, see \cite{cercignani1988} and references therein.\\
The collision operator in the BGK equation is given by $\mathcal{Q}(f) = (\mathcal{M}[f]-f)$, where $\mathcal{M}$ is a Maxwellian function, generally depending on the macroscopic moments \cite{BGK1954, perthame1989}. This can be explicitly written as
\begin{equation}
\label{eq::Maxwellian}
\mathcal{M}[f](t,x,v) := \frac{\rho}{(2 \pi T)^{d/2}} \exp \left ( - \frac{|v- u|^2}{2 T} \right ),
\end{equation}
where the \textit{density}, \textit{velocity} and	\textit{temperature} of the gas $\rho$, $u$ and $T$  are computed as suitable moments of the distribution function $f$,  
\begin{equation}
\label{eq:macroscopic_variables}
\rho = \int_{\R^d}f(v)\,dv, \quad u = \frac{1}{\rho}\int_{\R^d}v f(v) \, dv, 
\quad T = \frac{1}{d \,\rho} \int_{\R^d} \vert u - v \vert^2 f(v) \,dv.
\end{equation}
By construction, the collision operator preserves mass, momentum and kinetic energy. Indeed, it holds
  \begin{equation*}
    \int_{\R^{d}} \mathcal{Q}(f)(v) \, dv = 0, \quad  \int_{\R^{d}} \mathcal{Q}(f)(v) \, v \, dv = 0, \quad \int_{\R^{d}} \mathcal{Q}(f)(v) \, |v|^2 \, dv = 0.
  \end{equation*}

In the following, we will assume the initial condition $f_0$ has finite second spatial moment,
\begin{equation} \label{eq:rapidly_decaying}
    \int_{\mathbb{R}^d \times \mathbb{R}^d} f_0(x,v) |x|^2 \, dx  dv < \infty,
\end{equation}
which enforces a rapid spatial decay for the solution $f$ as $|x| \to \infty$. This decaying behavior is rigorously propagated in time, ensuring that the spatial moment of the solution $f(t,x,v)$ remains bounded for all $t > 0$, thereby guaranteeing that the solution $f(t,x,v) \to 0$ as $|x| \to \infty$.

\begin{rem}
While a diffusion term in velocity $\Delta_v f$ can appear in kinetic theory, often modeling a \textit{thermal bath} \cite{bobylev2002, gamba2004, rey2013, tristani2016}, the inclusion of spatial diffusion directly into the kinetic transport equation - similar to the diffusive Boltzmann models studied by Abramov \cite{abramov2017} - creates a mathematically distinct phase-space coupling. 
\end{rem}

\section{Convergence to the Viscous Euler Equations}
\label{Section:Convergence_Rigorous}

This section is devoted to the proof of convergence of the sequence of macroscopic moments
$(\rho^\eps, u^\eps, T^\eps)$ associated with the solution $f^\eps$ of the viscous BGK model
\eqref{eq:BGK_viscous} towards a solution of the compressible Euler equations with viscous
regularization, namely
\begin{equation}
\label{eq:Euler_viscous}
    \begin{cases}
        \partial_t \rho + \nabla_x \cdot (\rho u) = \tau \Delta_x \rho,\\
        \partial_t (\rho u) + \nabla_x \cdot (\rho u \otimes u + \mathrm{p}\,\mathbb{I})
            = \tau \Delta_x (\rho u),\\
        \partial_t (\rho E) + \nabla_x \cdot (u(\rho E + \mathrm{p}))
            = \tau \Delta_x (\rho E).
    \end{cases}
\end{equation}

The viscous regularization of the compressible Euler equations has been widely studied to justify the selection of physically admissible shock waves of the inviscid version (\textit{i.e.} $\tau = 0$) and to establish the well-posedness of entropy solutions. In the pioneering work of DiPerna \cite{diperna1982}, rigorous convergence of a viscous approximation of the 1D isentropic Euler equations to the correct inviscid limit has been proved. In \cite{bianchinibressan2005}, Bianchini and Bressan proved that for 1D strictly hyperbolic systems with small total variation, the vanishing viscosity method converges to the unique entropy-admissible solution, finally validating the concept of viscous regularization. In other words, in the one dimensional case, this viscous regularization is necessary to suppress non-physical oscillations and guarantee convergence to the unique physically admissible entropy solutions of the inviscid Euler equation for $\tau \to 0$. The main idea is that the viscosity solutions should coincide precisely with the limits of vanishing viscosity approximations, see also \cite{dafermos2016} and references therein.

In this section, we will prove convergence for a fixed $\tau >0$. The argument is structured as follows. We first derive the entropy dissipation inequality in Proposition \ref{prop:H_Theorem}, since it is a formal identity that holds for any sufficiently regular solution and its proof requires no compactness theory. The resulting a priori bounds are then used both to establish global existence in Proposition \ref{prop:existence_BGK} and to carry out the compactness argument that identifies the limit in Theorem \ref{thm:convergence_viscous_Euler}.

\subsection{Entropy Dissipation Inequality}

\begin{defi}[Boltzmann Entropy]
Let $f$ be a solution to the viscous BGK equation \eqref{eq:BGK_viscous}. The \textit{Boltzmann
$H$-functional} is defined as
\begin{equation}\label{def:H_functional}
    \mathcal{H}(f) = \int_{\R^d\times\R^d} f\log f\, dx\, dv.
\end{equation}
\end{defi}

\begin{prop}
\label{prop:H_Theorem}
Let $f^\eps$ be a sufficiently regular non-negative solution of the viscous BGK equation
\eqref{eq:BGK_viscous} satisfying \eqref{eq:rapidly_decaying}. Then it holds
\begin{equation}
\label{eq:entropy_inequality}
    \mathcal{H}(f^\eps(t))
    + \frac{1}{\eps}\int_0^t \mathcal{D}_{\mathrm{BGK}}(s)\, ds
    + 4\tau\int_0^t\!\!\int_{\R^d\times\R^d} |\nabla_x\sqrt{f^\eps(s)}|^2\, dx\, dv\, ds
    \le \mathcal{H}(f_0^\eps),
\end{equation}
where the BGK entropy dissipation is
\begin{equation}
\label{eq:D_BGK}
    \mathcal{D}_{\mathrm{BGK}}(t)
    := \int_{\R^d\times\R^d}
        (f^\eps - \mathcal{M}[f^\eps])\log\!\left(\frac{f^\eps}{\mathcal{M}[f^\eps]}\right)
        dx\, dv \ge 0.
\end{equation}
\end{prop}

\begin{proof}
Multiplying \eqref{eq:BGK_viscous} by $(1 + \log f^\eps)$ and integrating over
$(x,v) \in \R^d\times\R^d$, we obtain 
\begin{multline*}
    \int_{\R^d\times\R^d} (\partial_t f^\eps)(1 + \log f^\eps) +  (v\cdot\nabla_x f^\eps)(1 + \log f^\eps)\, dx\, dv \\= \frac{1}{\eps}\int_{\R^d\times\R^d}(\mathcal{M}[f^\eps]-f^\eps)(1+\log f^\eps)\, dx\, dv + \tau\int_{\R^d\times\R^d} (\Delta_x f^\eps)(1+\log f^\eps)\, dx\, dv,
\end{multline*}
where we treat each term separately.

\medskip
\noindent\textit{Transport term.}
Since $v\cdot\nabla_x(f\log f) = v\cdot\nabla_x f\,(1 + \log f)$,
\begin{equation*}
    \int_{\R^d\times\R^d} (v\cdot\nabla_x f^\eps)(1 + \log f^\eps)\, dx\, dv
    = \int_{\R^d\times\R^d} \nabla_x\cdot(v\, f^\eps\log f^\eps)\, dx\, dv = 0,
\end{equation*}
which vanishes by the divergence theorem under rapidly decaying behavior \eqref{eq:rapidly_decaying}.

\medskip
\noindent\textit{Viscosity term.}
Integrating by parts in $x$ and recalling that boundary terms vanish as above,
\begin{multline*}
    -\tau\int_{\R^d\times\R^d} (\Delta_x f^\eps)(1+\log f^\eps)\, dx\, dv
    = \tau\int_{\R^d\times\R^d} \nabla_x f^\eps\cdot\nabla_x\log f^\eps\, dx\, dv
    \\ = \tau\int_{\R^d\times\R^d} \frac{|\nabla_x f^\eps|^2}{f^\eps}\, dx\, dv.
\end{multline*}
Using the identity $\nabla_x f^\eps = 2\sqrt{f^\eps}\,\nabla_x\sqrt{f^\eps}$, we rewrite this as
\begin{equation}
\label{eq:fisher_information}
    \tau\int_{\R^d\times\R^d} \frac{|\nabla_x f^\eps|^2}{f^\eps}\, dx\, dv
    = 4\tau\int_{\R^d\times\R^d} |\nabla_x\sqrt{f^\eps}|^2\, dx\, dv \ge 0.
\end{equation}
This will be referred to as the \textit{spatial Fisher information} of $f^\eps$. It is non-negative and represents the \textit{a priori} bound provided by the viscous term $\tau\Delta_x f^\eps$.

\medskip
\noindent\textit{Relaxation term.}
We decompose $\log f^\eps = \log(f^\eps/\mathcal{M}[f^\eps]) + \log\mathcal{M}[f^\eps]$ and write
\begin{align*}
    &\frac{1}{\eps}\int_{\R^d\times\R^d}(\mathcal{M}[f^\eps]-f^\eps)(1+\log f^\eps)\, dx\, dv \\
    &\quad = -\frac{1}{\eps}\int_{\R^d\times\R^d}
        (f^\eps - \mathcal{M}[f^\eps])\log\!\frac{f^\eps}{\mathcal{M}[f^\eps]}\, dx\, dv \leq 0,
\end{align*}
where the equality holds since $\log\mathcal{M}[f^\eps]$ is a linear combination of the collision invariants $\phi(v) \in \{1, v, |v|^2\}$. Since $f^\eps$ and $\mathcal{M}[f^\eps]$ share the same macroscopic moments, we have $\int_{\R^d} (\mathcal{M}[f^\eps] - f^\eps)\phi(v)\, dv = 0$ for any collision invariant. The remaining integral defines $-\mathcal{D}_{\mathrm{BGK}}/\eps$, which is non-positive due to the classical elementary inequality $(a - b)\log(a/b) \ge 0$ for all $a, b > 0$.

\medskip
Collecting the three contributions and integrating in time yields \eqref{eq:entropy_inequality}.
\end{proof}

\subsection{Uniform A Priori Estimates}
\label{subsec:apriori}

The entropy inequality \eqref{eq:entropy_inequality} contains three non-negative terms on the
left-hand side: the entropy $\mathcal{H}(f^\eps(t))$ at time $t$, the relaxation dissipation term $\mathcal{D}_{\text{BGK}}$ and the spatial Fisher information, both integrated in time. To extract uniform bounds, we must assume that the initial condition $f_0^\eps$ satisfies
\begin{equation*}
    \int f_0^\eps(1 + |v|^2 + |x|^2 + |\log f_0^\eps|)\, dx\, dv < \infty.
\end{equation*}

\subsubsection{Control of the Entropy Lower Bound}
As done in \cite[Proposition 5.1]{alonso2026}, we can use the generalized Young's inequality 
\begin{equation*}
    xy \le x \log x - x + e^y, \qquad x \geq 0, \, y \in \R
\end{equation*} 
with $x=z$ and $y = -|v|^2/2 -|x|^2/2 + \log(\kappa/2)$ to obtain
\begin{equation*}
    z|\log z| \le z\log z + C_\kappa(z + z|v|^2 + z|x|^2) + \kappa\, e^{-|v|^2/2-|x|^2/2},
    \qquad \forall\, z \ge 0,\quad \kappa > 0.
\end{equation*}
Since the non-negativity of the solution $f^\eps \ge 0$ is strictly preserved by the evolution, as we will rigorously establish in Proposition \ref{prop:existence_BGK}, $\log f^\eps$ is well-defined. Integrating over the phase space, assuming finite mass $M_0$, finite energy $E_0$ and finite second order spatial moment $I_0$, we formally obtain
\begin{equation}
\label{eq:H_lower_bound}
    0 \leq \int_{\R^d\times\R^d} f^\eps|\log f^\eps|\, dx\, dv
    \le \mathcal{H}(f^\eps) + C(M_0, E_0, I_0, T),
\end{equation}
which implies
\begin{equation}
\label{eq:H_bounded_below}
    \mathcal{H}(f^\eps(t)) \ge -C(M_0, E_0, I_0, T).
\end{equation}
The constant $C(M_0, E_0, I_0, T)$ can be explicitly determined integrating the inequality 
\begin{equation*}
    f^\eps |\log f^\eps| \le f^\eps \log f^\eps + C_\kappa(f^\eps + f^\eps |v|^2 + f^\eps |x|^2) + \kappa\, e^{-|v|^2/2-|x|^2/2}
\end{equation*}
over the phase space $\mathbb{R}^d \times \mathbb{R}^d$, and using the total mass $M_0 = \iint f_0^\eps \,dx\,dv$ (constant in time), the total energy $E_0 = \frac{1}{2}\iint f_0^\eps |v|^2 \,dx\,dv$ (constant in time) and the total second order spatial moment $I = \frac{1}{2}\iint f^\eps |x|^2 \,dx\,dv$. Applying Grönwall's lemma, we get
\begin{equation*}
    C(M_0, E_0, I_0, T) = C_\kappa \Big(M_0 + 2E_0 + 2 \sup_{t \in [0,T]} I(t)\Big) + \kappa (2\pi)^d.
\end{equation*}
This time-dependent bound preserves \eqref{eq:H_bounded_below} globally in time.

\subsubsection{Uniform Bounds}

Combining \eqref{eq:entropy_inequality} with \eqref{eq:H_bounded_below}, we obtain the following
bounds, uniform in $\eps$.
\begin{enumerate}
    \item \textit{Mass and energy.} Integrating \eqref{eq:BGK_viscous} against $1$ and $|v|^2/2$
    and using that both are collision invariants together with the rapidly decaying boundary conditions \eqref{eq:rapidly_decaying} in $x$,
    \begin{equation}
    \label{eq:mass_energy_conservation}
        \frac{d}{dt}\int f^\eps\, dx\, dv = 0, \qquad
        \frac{d}{dt}\int f^\eps \frac{|v|^2}{2}\, dx\, dv = 0.
    \end{equation}

    \item \textit{Entropy.} From \eqref{eq:entropy_inequality} and \eqref{eq:H_bounded_below},
    \begin{equation}
    \label{eq:entropy_bound}
        \sup_{t\in[0,T]}\mathcal{H}(f^\eps(t)) \le \mathcal{H}(f_0^\eps) < \infty,
    \end{equation}
    and by \eqref{eq:H_lower_bound}, $\sup_t\int f^\eps|\log f^\eps|\, dx\, dv$ is uniformly
    bounded. This gives equi-integrability property and by the Dunford-Pettis Theorem \cite[Theorem 27-II]{dellacherie1978}, the sequence $\{f^\eps\}$ is weakly relatively compact in $L^1_\mathrm{loc}$.

    \item \textit{Fisher information.} From \eqref{eq:entropy_inequality},
    \begin{equation}
    \label{eq:fisher_bound}
        4\tau\int_0^T\!\!\int_{\R^d\times\R^d}|\nabla_x\sqrt{f^\eps}|^2\, dx\, dv\, dt
        \le \mathcal{H}(f_0^\eps) + C(M_0, E_0) < \infty,
    \end{equation}
    under the assumption of initial condition having finite moments and finite entropy.
\end{enumerate}

\subsection{Propagation of Moments}

We establish that high-order velocity moments in $L^1$ remain bounded uniformly in $\eps > 0$ over $t \in [0,T]$.

\begin{prop} \label{prop:propagation_moments}
Assume that the initial distribution satisfies $|v|^q f_0^\eps \in L^1(\mathbb{R}^d \times \mathbb{R}^d)$ for $q \ge 2$ uniformly in $\eps$. Then, any non-negative solution $f^\eps$ to \eqref{eq:BGK_viscous} satisfies
\begin{equation*}
    \sup_{\eps > 0} \, \sup_{0 \le t \le T} \| |v|^q f^\eps(t, \cdot, \cdot) \|_{L^1(\mathbb{R}^d \times \mathbb{R}^d)} \le C_T,
\end{equation*}
where $C_T$ depends only on the initial bounds and $T$, independently of $\eps$.
\end{prop}

\begin{proof}
Let $Y_q(t)$ denote the $L^1$ velocity moment of order $q$, namely
\begin{equation*}
    Y_q(t):= \int_{\mathbb{R}^d \times \mathbb{R}^d} f^\eps(t,x,v) |v|^q \,dx\,dv
\end{equation*}
Multiplying \eqref{eq:BGK_viscous} by $|v|^q$ and integrating over phase space, the transport and spatial diffusion terms vanish using \eqref{eq:rapidly_decaying}
\begin{equation*}
    \frac{d}{dt} Y_q(t) + \frac{1}{\eps} Y_q(t) = \frac{1}{\eps} M_q(t),
\end{equation*}
where $M_q(t) := \int_{\mathbb{R}^d \times \mathbb{R}^d} \mathcal{M}[f^\eps(t,x,v)] |v|^q \,dx\,dv$. Integrating using the factor $e^{t/\eps}$ yields the representation
\begin{equation} \label{eq:mild_moment_representation}
    Y_q(t) = Y_q(0) e^{-t/\eps} + \int_0^t \frac{1}{\eps} e^{-(t-s)/\eps} M_q(s) \,ds.
\end{equation}
Following \cite[Proposition 2.1]{zhang2007}, $M_q(s)$ is bounded by the macroscopic density, velocity, and temperature as
\begin{equation*}
    \int \mathcal{M}[f^\eps] |v|^q \,dv \le C_q \rho^\eps (|u^\eps|^q + (T^\eps)^{q/2}).
\end{equation*}
For $q > 2$, applying moment interpolation (see e.g. \cite{perthame1989, zhang2007}), the Maxwellian moments $M_q(t)$ can be controlled by the kinetic moments $Y_q(t)$ in $L^1_x$, ensuring $\sup_{s \in [0,T]} M_q(s) \le C_\theta + \theta Y_q (s)$ independent of $\eps$ for some constant $\theta \in (0,1)$. Since
\begin{equation*}
    \int_0^t \frac{1}{\eps} e^{-(t-s)/\eps} \,ds \le 1,
\end{equation*} 
substituting in \eqref{eq:mild_moment_representation} gives
\begin{equation*}
    Y_q(t) \le Y_q(0) e^{-t/\eps} + \left(1 - e^{-t/\eps}\right) \left( C_\theta + \theta \sup_{0 \le s \le t} Y_q(s) \right).
\end{equation*}
Taking the supremum over $t \in [0,T]$ yields $\sup_{0 \le t \le T} Y_q(t) \le C_T$ uniformly in $\eps > 0$, which completes the proof.
\end{proof}

\subsection{Macroscopic Spatial Regularity via Fisher Information}
\label{subsec:spatial_reg}

The goal is transferring the results on the uniform bound \eqref{eq:fisher_bound} on $\nabla_x\sqrt{f^\eps}$ to the macroscopic moments. Indeed, since there is no regularity in $v$, the Rellich-Kondrachov Theorem \cite[Section 5.7]{evans1998} cannot be applied directly to $f^\eps$ in phase space.

\begin{prop}
\label{prop:spatial_regularity_fisher_bounds}
Let $\psi(v)$ be a polynomial with $|\psi(v)| \le C(1+|v|^q)$ for some $q \ge 0$. Assume the moments of $f_0^\eps$ up to order $2q$ are uniformly bounded in $L^1(\R^d \times \R^d)$ and that $f_0^\eps$ has finite entropy $\mathcal{H}(f_0^\eps) < \infty$. Then
\begin{equation}
\label{eq:pointwise_gradient_bound}
    |\nabla_x \mu_\psi^\eps(t,x)|
    \le 2\!\left(\int_{\R^d}|\psi(v)|^2 f^\eps\, dv\right)^{\!1/2}
         \!\left(\int_{\R^d}|\nabla_x\sqrt{f^\eps}|^2\, dv\right)^{\!1/2},
\end{equation}
where $\mu_\psi^\eps(t,x) := \int_{\R^d}\psi(v)f^\eps(t,x,v)\, dv$. In particular,
\begin{equation}
\label{eq:moment_gradients_L2L1}
    \nabla_x\rho^\eps,\quad
    \nabla_x(\rho^\eps u^\eps),\quad
    \nabla_x(\rho^\eps E^\eps)
    \;\in\; L^2(0,T;\, L^1(\R^d)).
\end{equation}
\end{prop}

\begin{proof}
Using the identity $\nabla_x f^\eps = 2\sqrt{f^\eps}\,\nabla_x\sqrt{f^\eps}$, we apply the Cauchy-Schwarz inequality to obtain
\begin{multline*}
    |\nabla_x \mu_\psi^\eps|
    = \left|\int_{\R^d}\psi(v)\,\nabla_x f^\eps\, dv\right|
    = \left|2\int_{\R^d}\psi(v)\sqrt{f^\eps}\,\nabla_x\sqrt{f^\eps}\, dv\right|
    \\ \le 2\,\left(\int|\psi|^2 f^\eps\, dv\right)^{1/2}\,\,
         \left(\int|\nabla_x\sqrt{f^\eps}|^2\, dv\right)^{1/2}.
\end{multline*}
The first factor requires bounds on the higher-order moments. Integrating \eqref{eq:pointwise_gradient_bound} with respect to $x$ over $\R^d$ and applying the Cauchy-Schwarz inequality in $x$, we obtain
\begin{equation*}
    \|\nabla_x \mu_\psi^\eps\|_{L^1(\R^d)} 
    \le 2 \left(\int_{\R^d \times \R^d}|\psi(v)|^2 f^\eps\, dx\, dv\right)^{\!1/2}
         \left(\int_{\R^d \times \R^d}|\nabla_x\sqrt{f^\eps}|^2\, dx\, dv\right)^{\!1/2}.
\end{equation*}
Since moments of $f^\eps$ are uniformly bounded up to order $2q$ thanks to Proposition \ref{prop:propagation_moments}, the first integral is uniformly bounded in $L^\infty(0,T)$. Taking the $L^2(0,T)$ norm of both sides and applying the Fisher information bound \eqref{eq:fisher_bound} to the second factor, we conclude that $\nabla_x \mu_\psi^\eps \in L^2(0,T; L^1(\R^d))$ uniformly in $\eps$, yielding \eqref{eq:moment_gradients_L2L1}. Thus, we conclude that the macroscopic gradients $\mu_\psi$ are uniformly bounded in $L^1$. It is important to underline that the presence of the viscous term with $\tau >0$ is essential for the proof.
\end{proof}

\subsection{Time Compactness}
\label{sec:time_compact}
Let us integrate the viscous BGK equation \eqref{eq:BGK_viscous} against the collision invariants $\psi = 1, v_i, |v|^2$. Using the definitions \eqref{eq:macroscopic_variables}, we obtain the macroscopic conservation laws
\begin{equation}
\label{eq:macro_system}
    \partial_t U^\eps + \nabla_x\cdot\mathbb{F}^\eps = \tau\Delta_x U^\eps,
\end{equation} 
where $U^\eps = (\rho^\eps, \rho^\eps u^\eps, \rho^\eps E^\eps)^T$ and $\mathbb{F}^\eps$ is the corresponding flux tensor.\\ 
Let us observe that the mass and momentum fluxes are globally bounded by the energy, while the energy flux involves the third-order velocity moment $\int v |v|^2 f^\eps\, dv$. By applying Proposition \ref{prop:propagation_moments}, the uniform bound on the third moment of the initial data $f_0^\eps$ guarantees that the third moment of $f^\eps(t)$ remains bounded locally in space, ensuring that the macroscopic fluxes satisfy $\mathbb{F}^\eps \;\in\; L^1(0,T;\, L^1_\mathrm{loc}(\R^d))$.\\
From Proposition \ref{prop:spatial_regularity_fisher_bounds}, we also have $U^\eps \in L^2(0,T; W^{1,1}_\mathrm{loc})$, which implies that the diffusion term satisfies $\tau\Delta_x U^\eps = \tau\nabla_x\cdot(\nabla_x U^\eps) \in L^2(0,T; W^{-1,1}_\mathrm{loc})$. Since $\nabla_x\cdot\mathbb{F}^\eps \in L^1(0,T; W^{-1,1}_\mathrm{loc})$, we have
\begin{equation*}
    \partial_t U^\eps \;\in\; L^1(0,T;\, W^{-1,1}_\mathrm{loc}(\R^d)).
\end{equation*}

We finally apply the Aubin-Lions-Simon compactness theorem \cite{simon1987} using the continuous embeddings
\begin{equation*}
    W^{1,1}_\mathrm{loc} \;\hookrightarrow\hookrightarrow\; L^1_\mathrm{loc}
    \;\hookrightarrow\; W^{-1,1}_\mathrm{loc}.
\end{equation*}
Since $U^\eps$ is uniformly bounded in $L^2(0,T; W^{1,1}_\mathrm{loc})$ and $\partial_t U^\eps$ is uniformly bounded in $L^1(0,T; W^{-1,1}_\mathrm{loc})$, the family $\{U^\eps\}$ is strongly relatively compact in $L^1_\mathrm{loc}((0,T)\times\R^d)$. 
\begin{prop}[Strong Compactness of Macroscopic Variables]
\label{prop:time_compactness_macro}
Let $\{f^\eps\}_{\eps > 0}$ be a family of non-negative weak solutions to the viscous BGK equation \eqref{eq:BGK_viscous} satisfying the uniform bounds of Proposition~\ref{prop:spatial_regularity_fisher_bounds}. Then, the corresponding sequence of macroscopic variables $U^\eps = (\rho^\eps, \rho^\eps u^\eps, \rho^\eps E^\eps)^T$ is strongly relatively compact in $L^1_\mathrm{loc}((0,T) \times \R^d)$.
\end{prop}

\subsection{Global Existence}

The main objective of this section is to prove the existence of weak solutions to the viscous BGK equation \eqref{eq:BGK_viscous}. We follow the approach introduced by Perthame \cite{perthame1989, perthame1993}, inspired by \cite{dipernalions1989} and we adapt it to our framework.

\begin{prop}
\label{prop:existence_BGK}
Let $f_0^\eps \ge 0$ satisfy $\int f_0^\eps(1 + |v|^2 + |\log f_0^\eps|)\, dx\, dv < \infty$.
For each fixed $\eps > 0$ and $\tau > 0$, there exists a global weak solution $f^\eps \ge 0$ of
the viscous BGK model \eqref{eq:BGK_viscous} satisfying the conservation of mass and total energy
\eqref{eq:mass_energy_conservation}, and the entropy dissipation inequality \eqref{eq:entropy_inequality}.
\end{prop}

\begin{proof}
\noindent\textit{Step 1: Regularization.}
Fix $\delta > 0$. Define the regularized Maxwellian $\mathcal{M}_\delta$ by replacing the temperature with $T_\delta := \max\{T,\delta\}$ and the macroscopic velocity with $u_\delta := u \min(1, \frac{1}{\delta|u|})$. This prevents degeneracy of the Maxwellian, ensuring that all integrands are bounded uniformly, as done in \cite{perthame1989}.

\medskip
\noindent\textit{Step 2: Fundamental solution and Duhamel representation.}
The homogeneous operator $\partial_t + v\cdot\nabla_x - \tau\Delta_x$ describes advection along the characteristics $x(t) = x_0 + vt$, coupled with spatial diffusion having rate $\tau>0$, whose fundamental solution is given by
\begin{equation}
\label{eq:G_tau}
    \Gamma_\tau(t, x, v) := \frac{1}{(4\pi\tau t)^{d/2}}
        \exp\!\left(-\frac{|x - vt|^2}{4\tau t}\right), \qquad t > 0,\; x, v \in \mathbb{R}^d.
\end{equation}
Let us consider $g \in L^\infty(0,T;\, L^1_2 \cap L^2(\mathbb{R}^d\times\mathbb{R}^d))$, where the space $L^1_2$ is weighted by $\langle v \rangle^2 = 1 + |v|^2$. Applying Duhamel's principle to the source term $\mathcal{M}_\delta[g]/\varepsilon$ yields the fixed-point map $\mathcal{T}(g) = f$, expressed as
\begin{multline}
\label{eq:Duhamel_BGK}
    f(t, x, v) = e^{-t/\varepsilon}\!\int_{\mathbb{R}^d} \Gamma_\tau(t, x-y, v)\, f_0^\varepsilon(y, v)\, dy \\ 
    + \frac{1}{\varepsilon}\int_0^t e^{-(t-s)/\varepsilon}\!\int_{\mathbb{R}^d} \Gamma_\tau(t-s, x-y, v)\,\mathcal{M}_\delta[g](s, y, v)\, dy\, ds.
\end{multline}
Since the kernel $\Gamma_\tau \ge 0$, the Maxwellian $\mathcal{M}_\delta \ge 0$, and the initial data $f_0^\varepsilon \ge 0$, this representation ensures that $f \ge 0$, in the same spirit of the result established by Perthame \cite{perthame1989} for the inviscid case.

\medskip
\noindent\textit{Step 3: Fixed-point argument.}
Let
\begin{equation*}
    X := C([0,T]; L^1_2 \cap L^2(\mathbb{R}^d \times \mathbb{R}^d)),
\end{equation*}
endowed with the norm 
$||\cdot||_X = ||\cdot||_{L^1_2} + ||\cdot||_{L^2}$.
As observed in \cite[Proposition 5]{perthame1989}, the regularization provided by $T_\delta$ and $u_\delta$ ensures the Maxwellian mapping is Lipschitz continuous in the weighted space. Since $\mathcal{M}_\delta$ is Lipschitz continuous and the fundamental solution $\Gamma_\tau$ preserves the $L^1_2 \cap L^2$ norm, the operator $\mathcal{T}(g)$ is a contraction map on the space $X$ for small times. Banach's fixed-point theorem \cite[Section 9.2.1]{evans1998} then yields a fixed point $f_\delta^\eps \in X$ satisfying
\eqref{eq:Duhamel_BGK}, with $g = f_\delta^\eps$. Since the Lipschitz bounds depend linearly on conserved quantities, the local in time solution can be iterated over the whole time interval $[0, T]$.

\medskip
\noindent\textit{Step 4: Passage to the limit $\delta \to 0$.}
Since $\{f_\delta^\eps\}_\delta$ satisfies uniform bounds in $\eps$ and $\delta$ and mass and energy are conserved, the family $\{f_\delta^\eps\}$ is equi-integrable by \eqref{eq:entropy_bound}, and the Fisher information \eqref{eq:fisher_bound} is uniformly bounded. By Dunford-Pettis Theorem \cite[Theorem 27-II]{dellacherie1978}, a subsequence of $\{f_\delta^\eps\}$ converges weakly in $L^1_\mathrm{loc}$ to some $f^\eps \ge 0$.\\
To identify the limit of the nonlinear term $\mathcal{M}_\delta[f_\delta^\eps]$, we apply the Aubin-Lions-Simon theorem to the macroscopic moments $U_\delta^\eps = (\rho_\delta^\eps, \rho_\delta^\eps u_\delta^\eps, \rho_\delta^\eps E_\delta^\eps)$. Applying Proposition \ref{prop:spatial_regularity_fisher_bounds} and Proposition~\ref{prop:time_compactness_macro} to $f_\delta^\eps$, we obtain that the sequence $\{U_\delta^\eps\}_\delta$ is strongly relatively compact in $L^1_\mathrm{loc}((0,T) \times \R^d)$. This yields strong convergence, obtaining the macroscopic equation
\begin{equation*}
    \partial_t U_\delta^\eps = -\nabla_x\cdot\mathbb{F}_\delta^\eps + \tau\Delta_x U_\delta^\eps + \mathcal{O}(\delta).
\end{equation*}
Such strong convergence result guarantees the convergence almost everywhere of the macroscopic quantities $(\rho_\delta^\eps, u_\delta^\eps, T_\delta^\eps)$ to $(\rho^\eps, u^\eps, T^\eps)$. We establish that the regularized source term converges strongly to the Maxwellian
\begin{equation*}
    \mathcal{M}_\delta[f_\delta^\eps] \longrightarrow \mathcal{M}[f^\eps] \quad \text{strongly in } L^1_\mathrm{loc}((0,T) \times \R^d \times \R^d).
\end{equation*}
Combining the weak convergence of the linear transport-diffusion terms with the strong convergence of the relaxation operator, we can pass to the limit $\delta \to 0$ in the weak formulation of the regularized problem. We conclude that $f^\eps$ is a global weak solution satisfying the viscous BGK equation \eqref{eq:BGK_viscous}.
\end{proof}

\subsection{Convergence to the Viscous Euler System}

\begin{thm}
\label{thm:convergence_viscous_Euler}
Let $f_0^\eps \ge 0$ satisfy $\int f_0^\eps(1 + |x|^2 + |v|^q + |\log f_0^\eps|)\, dx\, dv < \infty$ for all $q \le 4$. There exists a subsequence such that
\begin{equation}
    (\rho^\eps,\, \rho^\eps u^\eps,\, \rho^\eps E^\eps)
    \;\longrightarrow\; (\rho,\, \rho u,\, \rho E)
    \quad \text{strongly in } L^1_\mathrm{loc}((0,T)\times\R^d).
\end{equation}
Consequently $u^\eps \to u$ and $T^\eps \to T$ almost everywhere on $\{\rho^\eps > 0\}$, and the limit
$(\rho, u, T)$ is a weak solution of the viscous Euler system \eqref{eq:Euler_viscous}.
\end{thm}

\begin{proof}
The strong convergence of the moments follows from the Aubin-Lions-Simon argument in Section \ref{sec:time_compact}. The goal is to identify the limit equation of each term of the weak formulation \eqref{eq:macro_system}.

From the entropy dissipation estimate \eqref{eq:entropy_inequality} and the lower bound \eqref{eq:H_bounded_below}, under the assumption of finite moments and finite entropy for the initial condition, we have
\begin{equation*}
    \frac{1}{\eps}\int_0^T \int_{\R^d} \mathcal{D}_{\mathrm{BGK}}(t,x)\, dx \, dt \le \mathcal{H}(f_0^\eps) + C(M_0, E_0) =: C_0 < \infty.
\end{equation*}
By the Csisz\'ar-Kullback-Pinsker inequality in $v$, we have
\begin{equation*}
    \|f^\eps(t,x,\cdot) - \mathcal{M}[f^\eps](t,x,\cdot)\|_{L^1_v}^2 \le 2\rho^\eps(t,x) \mathcal{D}_{\mathrm{BGK}}(t,x).
\end{equation*}
Integrating in $x$ and applying the Cauchy-Schwarz inequality yields
\begin{align*}
    \|f^\eps(t) - \mathcal{M}[f^\eps](t)\|_{L^1_{x,v}} 
    &\le \sqrt{2} \left( \int_{\R^d} \rho^\eps \,dx \right)^{\!1/2} \left( \int_{\R^d} \mathcal{D}_{\mathrm{BGK}} \,dx \right)^{\!1/2} \\
    &= \sqrt{2 M_0} \left( \int_{\R^d} \mathcal{D}_{\mathrm{BGK}}(t,x) \,dx \right)^{\!1/2}.
\end{align*}
The time integration gives 
\begin{equation*}
    \int_0^T\|f^\eps - \mathcal{M}[f^\eps]\|_{L^1_{x,v}}^2\, dt
    \le 2M_0 \int_0^T \int_{\R^d} \mathcal{D}_{\mathrm{BGK}}\, dx\, dt \le 2M_0 C_0 \eps \to 0.
\end{equation*}
Thus, $f^\eps - \mathcal{M}[f^\eps] \to 0$ strongly in $L^2(0,T; L^1_{x,v})$ and, using compactness result in Proposition \ref{prop:time_compactness_macro}, we get $f^\eps \to \mathcal{M}_{\rho,u,T}$ strongly in $L^1_\mathrm{loc}$.

We must now pass to the limit in the momentum flux $\mathbb{F}_v^\eps := \int v\otimes v\, f^\eps\, dv$ and the energy flux $\mathbb{F}_E^\eps := \int v \frac{|v|^2}{2} f^\eps\, dv$. Let us define the fluctuation $h^\eps = f^\eps - \mathcal{M}[f^\eps]$ and decompose the momentum flux as
\begin{equation*}
    \mathbb{F}_v^\eps = \int v\otimes v\, h^\eps\, dv + \int v\otimes v\,\mathcal{M}[f^\eps]\, dv.
\end{equation*}
For the first term, applying the Hölder inequality yields
\begin{equation*}
    \left\| \int_{\R^d} |v|^2 |h^\eps|\, dv \right\|_{L^1_{t,x}} 
    \le \|h^\eps\|_{L^1_{t,x,v}}^{1/2} \left( \int_0^T \int_{\R^d\times\R^d} |v|^4 |h^\eps|\, dv\, dx\, dt \right)^{\!1/2}.
\end{equation*}
Since the initial condition has a bounded 4th-order moment, the bound is uniformly propagated thanks to Proposition \ref{prop:propagation_moments}.\\
Since $\|h^\eps\|_{L^1} \to 0$, the first term vanishes. Similarly, for the energy flux $\mathbb{F}_E$, we get
\begin{equation*}
    \left\| \int_{\R^d} |v|^3 |h^\eps|\, dv \right\|_{L^1_{t,x}} 
    \le \|h^\eps\|_{L^1_{t,x,v}}^{1/4} \left( \int_0^T \int_{\R^d\times\R^d} |v|^4 |h^\eps|\, dv\, dx\, dt \right)^{\!3/4} \;\longrightarrow\; 0.
\end{equation*}

Therefore, the limits of the fluxes are entirely determined by the Maxwellian functions, which give
\begin{align*}
    \int v\otimes v\,\mathcal{M}[f^\eps]\, dv &= \rho^\eps u^\eps\otimes u^\eps + \rho^\eps T^\eps\,\mathbb{I}, \\
    \int v \frac{|v|^2}{2} \mathcal{M}[f^\eps]\, dv &= \left( \rho^\eps E^\eps + \rho^\eps T^\eps \right) u^\eps,
\end{align*}
both of which converge strongly in $L^1_\mathrm{loc}$ to their respective limits $\rho u\otimes u + \rho T\,\mathbb{I}$ and $(\rho E + \rho T)u$ by the strong convergence of the macroscopic moments proved in Proposition \ref{prop:time_compactness_macro}.\\
Collecting all limits and passing them into the weak formulation of the macroscopic conservation laws \eqref{eq:macro_system}, we identify the limit as the viscous Euler system \eqref{eq:Euler_viscous}.
\end{proof}

\section{Formal hydrodynamic limit: entropy production}
\label{Section:entropy_single}
The main goal of this section is investigating the formal hydrodynamic limit of the viscous BGK model using the Chapman-Enskog expansion. In particular, we would like to show that models derived within the framework of symmetric hyperbolic thermodynamically compatible (SHTC) theory can be obtained as limit of collisional kinetic equations.\\
The idea is to retrieve the Euler equations written in extended set of primitive variables and to automatically derive the entropy production term coming from the viscous regularization, without any further assumption. Unlike the symmetric hyperbolic and thermodynamically compatible (SHTC) theory \cite{godunov1961, ruggeri1981, romenski1998, thomann2023}, where the entropy production term is computed to enforce total energy conservation, in our framework the latter is automatically inherited from the structural properties of the BGK model. Indeed, in SHTC models the energy conservation law is considered as an extra conservation law, obtained via suitable combination of the governing equations.\\
The Euler equations with viscous regularization for the extended set of primitive variables $(\rho, \rho u, \rho s)$ read as follows
\begin{subnumcases}
        \displaystyle \partial_t \rho + \nabla_x \cdot (\rho u) = \tau \Delta_x \rho, \label{eq:viscous_Euler_density}\\
        \partial_t (\rho u) + \nabla_x \cdot (\rho u \otimes u + p \mathbb{I} ) = \tau \Delta_x (\rho u),\label{eq:viscous_Euler_momentum}\\
        \partial_t (\rho s) + \nabla_x \cdot ( \rho u s ) = \tau \Delta_x (\rho s) + \Pi. \label{eq:viscous_Euler_entropy}
\end{subnumcases}
The conservation of mass and the conservation of momentum have been already derived in Section \ref{Section:Convergence_Rigorous}, but we briefly recall the strategy using the Chapman-Enskog approach \cite{chapmancowling1939}. Let us first expand $f^\eps$ around a global Maxwellian function, namely
\begin{equation*}
    f^\eps = \mathcal{M}_{\rho,u,T} + \eps g^\eps.
\end{equation*} 
Now, integrating the viscous BGK equation \eqref{eq:BGK_viscous} against $\psi=1,v$ and using the expressions of the macroscopic fields given in \eqref{eq:macroscopic_variables}, we can truncate the equation at $0$-th order expansion in $\eps$. In particular, we easily obtain the mass conservation equation with viscous correction \eqref{eq:viscous_Euler_density} as
\begin{equation}
    \partial_t \rho + \nabla_x (\rho u) = \tau \Delta_x \rho,
\end{equation}
and the momentum conservation equation with viscous correction  \eqref{eq:viscous_Euler_momentum}, \textit{i.e.}
\begin{equation}
    \partial_t (\rho u) + \nabla_x (\rho u \otimes u + p \mathbf{I}) = \tau \Delta_x (\rho u).
\end{equation}
Let us now focus on the entropy inequality, which can be derived in the form
\begin{multline*}
    \partial_t \langle f^\eps, (-\log f^\eps -1) \rangle + \langle v \cdot \nabla_x f^\eps, (-\log f^\eps -1) \rangle \\ = \langle \mathcal{Q}(f^\eps), (-\log f^\eps -1) \rangle + \tau \langle \Delta_x f^\eps, (-\log f^\eps -1) \rangle,
\end{multline*}
where $\langle \cdot, \cdot \rangle$ denotes the $L^2$ inner product.
This leads to the following equation truncated at the $0$-th order expansion in $\eps$,
\begin{equation*}
    \partial_t (\rho s) + \nabla_x \cdot (\rho s u) = \tau \Delta_x (\rho s) + \tau \int_{\R^d} \frac{|\nabla_x \mathcal{M}_{\rho,u,T}|^2}{\mathcal{M}_{\rho,u,T}} dv.
\end{equation*}
where 
\begin{equation}
    \rho s : = -\int_{\R^d} \mathcal{M}_{\rho,u,T} \log \mathcal{M}_{\rho,u,T}\, dv.
\end{equation}
Note that in compressible flow theory, the thermodynamic entropy has opposite sign compared to the kinetic entropy (or Boltzmann H-functional) given in \eqref{def:H_functional}. \\
Let us finally compute the term 
\begin{equation*}
    \int_{\R^d} \frac{|\nabla_x \mathcal{M}_{\rho,u,T}|^2}{\mathcal{M}_{\rho,u,T}} dv,
\end{equation*}
which corresponds to the Fisher information \eqref{eq:fisher_information} evaluated in $\mathcal{M}_{\rho,u,T}$.\\
In the following we will express the Maxwellian function in terms of the specific entropy $s$ as $\mathcal{M}_{\rho, u, s}$. This change of variables is specifically tailored to exploit the thermodynamic structure of the system. It allows us to compute derivatives directly with respect to the entropy variables by substituting the relation $T=T(\rho, s)$, which is obtained by inverting the macroscopic definition of the entropy. Let us consider
\begin{equation*}
    \nabla_x \mathcal{M}_{\rho,u,s} = \nabla_q \mathcal{M}_{\rho,u,s} \cdot \nabla_x q,
\end{equation*}
where $q=(\rho, \rho u, \rho s)$ is the extended set of primitive variables. Then 
\begin{equation*}
    \frac{|\nabla_x \mathcal{M}_{\rho,u,s}|^2}{\mathcal{M}_{\rho,u,s}} = \frac{|\nabla_q \mathcal{M}_{\rho,u,s} \cdot \nabla_x q|^2}{\mathcal{M}_{\rho,u,s}}.
\end{equation*}
Using $\frac{\nabla_q g}{g} = \nabla_q (\log g)$, we can integrate over the velocity space and rewrite 
\begin{equation*}
    \int_{\R^d} \frac{|\nabla_x \mathcal{M}_{\rho,u,s}|^2}{\mathcal{M}_{\rho,u,s}} = (\nabla_x q)^T \left( \int_{\R^d} \mathcal{M}_{\rho,u,s} \nabla_q \log \mathcal{M}_{\rho,u,s} \otimes \nabla_q \log \mathcal{M}_{\rho,u,s} dv \right) \nabla_x q,
\end{equation*}
which is equivalent to 
\begin{equation*}
    \int_{\R^d} \frac{|\nabla_x \mathcal{M}_{\rho,u,s}|^2}{\mathcal{M}_{\rho,u,s}} = (\nabla_x q)^T H(q) \nabla_x q,
\end{equation*}
where the matrix $H(q)$ is defined as
\begin{equation}
\label{eq:formal_matrixH}
    H(q) : = \int_{\R^d} \mathcal{M}_{\rho,u,s} \nabla_q \log \mathcal{M}_{\rho,u,s} \otimes \nabla_q \log \mathcal{M}_{\rho,u,s} dv.
\end{equation}
In conclusion, we obtain the following equation for the entropy production, at the $0$-th order expansion
\begin{equation}
    \partial_t (\rho s) + \nabla_x \cdot (\rho s u) = \tau \Delta_x (\rho s) + \tau (\nabla_x q)^T H(q) \nabla_x q.
\end{equation}

In the remainder of this section we compare our approach with the SHTC methodology proposed in \cite{thomann2023}. Indeed, they recover as final form for the entropy production term $\Pi$ in \eqref{eq:viscous_Euler_entropy} the following expression
\begin{equation*}
    \Pi=\frac{1}{T} \nabla_q^2 \mathcal{E}.
\end{equation*}
The goal is to show that the two expressions are completely equivalent, since the entropy production term derived from the BGK model automatically provides the conservation of the total energy, without any further assumptions. Let us recall that the expression of the total energy at equilibrium is given by
\begin{equation}
    \mathcal{E} = \frac{1}{2} \int_{\R^d} |v|^2 f^\eps dv = \frac{1}{2} \int_{\R^d} |v|^2 \mathcal{M}_{\rho,u,s} dv.
\end{equation}
Let us compute 
\begin{equation}
    \partial_{q_k} \partial_{q_j} \mathcal{E} = \frac{1}{2} \int_{\R^d} |v|^2 \partial_{q_k} \partial_{q_j} \mathcal{M}_{\rho,u,s} dv, 
\end{equation}
by observing that
\begin{equation}
    |v|^2 = -2T \log \mathcal{M}_{\rho,u,s} + 2T \log \left(\frac{\rho}{(2\pi T)^{d/2}} \right) - |u|^2 + 2 u \cdot v.
\end{equation}
Then, defining $C(q) = T \log(\rho/(2\pi T)^{d/2})-|u|^2/2$ and $B(q)=u$, we obtain
\begin{multline}
    \partial_{q_k} \partial_{q_j} \mathcal{E} = \int_{\R^d} |v|^2 \partial_{q_k} \partial_{q_j} \mathcal{M}_{\rho,u,s} dv \\
    = -T \int_{\R^d} \log \mathcal{M}_{\rho,u,s} \partial_{q_k} \partial_{q_j} \mathcal{M}_{\rho,u,s} dv + C(q) \int_{\R^d} \partial_{q_k} \partial_{q_j} \mathcal{M}_{\rho,u,s} dv + B(q) \int_{\R^d} v \partial_{q_k} \partial_{q_j} \mathcal{M}_{\rho,u,s} dv \\
    = -T \int_{\R^d} \log \mathcal{M}_{\rho,u,s} \partial_{q_k} \partial_{q_j} \mathcal{M}_{\rho,u,s} dv,
\end{multline}
where we have used $\partial_{q_k} \partial_{q_j} (\rho) = 0$ and $\partial_{q_k} \partial_{q_j} (\rho u) = 0$, .\\
Let us recall that
\begin{multline*}
    \log \mathcal{M}_{\rho,u,s} (\partial_{q_k} \partial_{q_j} \mathcal{M}_{\rho,u,s}) \\ = - \partial_{q_j} \mathcal{M}_{\rho,u,s} \partial_{q_k} \log \mathcal{M}_{\rho,u,s} + \partial_{q_k} (\log \mathcal{M}_{\rho,u,s} \partial_{q_j} \mathcal{M}_{\rho,u,s}),
\end{multline*}
and 
\begin{multline*}
    \int_{\R^d}  \log \mathcal{M}_{\rho,u,s} (\partial_{q_k} \partial_{q_j} \mathcal{M}_{\rho,u,s}) dv = - \int_{\R^d} \partial_{q_k} \log \mathcal{M}_{\rho, u, s} \partial_{q_j} \mathcal{M}_{\rho,u,s} dv \\
    = - \int_{\R^d} \mathcal{M}_{\rho,u,s} \partial_{q_k} \log \mathcal{M}_{\rho, u, s} \partial_{q_j} \log \mathcal{M}_{\rho,u,s} dv.
\end{multline*}
Indeed, the integral term
\begin{multline*}
    \int_{\R^d} \partial_{q_j} \mathcal{M}_{\rho,u,s} \partial_{q_k} \log \mathcal{M}_{\rho,u,s} dv = \int_{\R^d} \partial_{q_j} \partial_{q_k} \left[\mathcal{M}_{\rho,u,s} \log \mathcal{M}_{\rho,u,s} -\mathcal{M}_{\rho,u,s} \right] dv \\ = \partial_{q_j} \partial_{q_k} \left[ \rho s - \rho \right] = 0.
\end{multline*}
Finally, we get
\begin{equation*}
    \partial_{q_k} \partial_{q_j} \mathcal{E} = T \int_{\R^d} \mathcal{M}_{\rho,u,s} \left(\partial_{q_k} \log \mathcal{M}_{\rho,u,s} \right)\left(\partial_{q_j} \log \mathcal{M}_{\rho,u,s} \right) dv = T H(q),
\end{equation*}
where $H(q)$ is the matrix defined in \eqref{eq:formal_matrixH}.

\begin{lem}
The matrix $H(q)$ defined in \eqref{eq:formal_matrixH} is positive definite.
\end{lem}

\begin{proof}
Let us first consider the Maxwellian in terms of the primitive variable $W=(\rho,u,T)$ and let us observe that the components of $\nabla_W \log \mathcal{M}$ are linearly independent as functions of the velocity variable $v$, since they are linear combinations of the collision invariants $1, v, |v|^2$.\\
Let $q=(\rho,\rho u,\rho s)$ be the extended set of primitive variables. In what follows, we strictly assume the absence of vacuum states, i.e. $\rho > 0$. Then, the mapping $q \mapsto W$ is a smooth diffeomorphism. Thus, 
\begin{equation}
\label{eq:prop_Hposdef_change}
    \nabla_q \log \mathcal{M} = \left(\frac{\partial W}{\partial q} \right)^T \nabla_W \log \mathcal{M},
\end{equation}
and the Jacobian matrix $J = \left(\tfrac{\partial W}{\partial q} \right)^T $ is invertible.\\
Now, let us consider $z \in \R^{d+2}$ be an arbitrary non-zero vector.
Using \eqref{eq:prop_Hposdef_change}, it is possible to rewrite 
\begin{equation}
    z^T H(q) z = \int_{\R^d} \mathcal{M} ((Jz) \cdot \nabla_W \log \mathcal{M})^2\,dv.
\end{equation}
Let $y=Jz$, which is non null, since $z \neq 0$ and $J$ is invertible. Moreover, the components of $\nabla_W \log \mathcal{M}$ are linearly independent, which implies that the dot product $y \cdot \nabla_W \log \mathcal{M}$ is a non-zero polynomial in $v$. Therefore, its square is strictly positive for almost every $v \in \R^d$ which, combined with the positivity of the Maxwellian function $\mathcal{M}$, yields
\begin{equation}
    \int_{\R^d} \mathcal{M} (y \cdot \nabla_W \log \mathcal{M})^2\,dv > 0.
\end{equation}
This concludes the proof.
\end{proof}

\section{Extension to the multi-species viscous BGK equation}
Several different kinetic models have been proposed in literature to treat gas mixtures, both using BGK-type operators \cite{andries2002, klingenberg2017, haack2017, bobylev2018}, Boltzmann-type operators \cite{chapmancowling1939, briant2016, reytenna2025} and mixed Boltzmann-BGK operators \cite{bisiboscheri2022, bisi2024}. In this work, we focus on a BGK-type model, mainly following \cite{klingenberg2017, haack2017}. \\ 
Consider a mixture of $M$ species, each described by a distribution function $f_p = f_p(t,x,v)$, where the mass of a molecule of component $p$ is denoted by $m_p$. Given $\x \in \Omega \in \R^{d}$, $\uv \in \R^{d}$, the distribution functions $f_p$ evolve according to the multi-species viscous BGK equation, which, in absence of external forces, takes the following form
\begin{equation}
\label{eq:BGK_multi_viscous}
\begin{cases}
    \partial_t f_p^\eps + v \cdot \nabla_x f_p^\eps = \displaystyle \frac{1}{\eps_{pp}} \left(\mathcal{M}_{p}[\f^\eps] - f_p^\eps \right) + \displaystyle \sum_{\substack{q=1\\q \neq p}}^M \frac{1}{\eps_{pq}} \left(\mathcal{M}_{pq}[\f^\eps] - f_p^\eps \right) + \tau \Delta_x f^\eps_p,\\
    f_p^\eps(0,x,v) = f^\eps_{0,p}(x,v),
\end{cases}
\end{equation}
where $\f=(f_1,\dots, f_M)$ and $\tau>0$. The parameter $\eps_{pq} >0$ denotes the relative \textit{Knudsen number} associated with collisions between particles of species $p$ and species $q$. The choice of defining the Knudsen number $\eps_{pq}$ for each type of collision between species $p$ and $q$ allows for a more accurate representation of multi-species kinetic effects \cite{degond1996, bisi2011multi, pupporeytenna2026}, especially in systems where inter-species interactions dominate over intra-species dynamics.\\ 
The formulation of the multispecies BGK operator requires the definition of the mixture Maxwellian functions, namely
\begin{equation}
    \mathcal{M}_{pq}[f_p, f_q] (t,x,v) := n_p \left(\frac{m_p}{2 \pi T_{pq}}\right)^{d/2} \exp \left ( - \frac{m_p |v- \bar{u}_{pq}|^2}{2 T_{pq}} \right ),
\end{equation}
where the macroscopic quantities are defined as suitable moments of the distribution functions $f_p$. In particular, the number density $n_p$, the mass density $\rho_p$, the average macroscopic velocity $\bar{u}_p$, the partial pressure $P_p$ and the temperature $\bar{T}_p$ for each species are computed as
\begin{equation}
\label{eq:macro_multi}
	n_p=\int_{\R^{d}} f_p(\uv) d\uv, \quad \rho_p \bar{u}_p=\int_{\R^{d}} m_p \uv f_p d\uv ,\quad \bar{T}_p=\frac{m_p}{d \, n_p} \int_{\R^{d}} |\uv-\bar{u}_p|^2 f_p d\uv,
\end{equation}
with $\rho_p= m_p\, n_p$ and $P_p= n_p \, \bar{T}_p$. We can also define the total moments for the mixture, given by
\begin{multline}
\label{eq:total_macro_quantities_multi}
	n= \sum_p n_p, \quad \rho= \sum_p \rho_p, \quad \rho\bar{u} = \sum_p \rho\bar{u}_p,\\ P=\sum_p P_p, \quad \bar{T}= \sum_p \frac{m_p}{d \, n_p} \int_{\R^{d}} |\uv-\bar{u}|^2 f_p d\uv.
\end{multline}
The quantities $\bar{u}_{pq}$ and $T_{pq}$ are mixture moments, which can be defined by imposing conservation of total momentum and conservation of total energy. We retrieve \cite{bobylev2018} the following expression for the mixture velocity
\begin{equation}
    \bar{u}_{pq} = \frac{\rho_p \bar{u}_p + \rho_q \bar{u}_q}{\rho_p + \rho_q},
\end{equation}
and the following expression for the mixture temperature
\begin{equation}
\label{eq:temperature_mixture}
    T_{pq} = \frac{n_p T_p + n_q T_q}{n_p + n_q} + \frac{\rho_p (\bar{u}_p^2 - \bar{u}_{pq}^2) + \rho_q (\bar{u}_q^2 - \bar{u}_{qp}^2)}{d (n_p + n_q)}.
\end{equation}
As observed in \cite[Section 2.7]{klingenberg2017}, this choice guarantees $T_{pq} \geq 0$ for all $p, q$. 

\subsection{Entropy Dissipation}
Let us investigate the long time behavior of solutions to the viscous multispecies BGK equation \eqref{eq:BGK_multi_viscous}. For the sake of simplicity and without loss of generality, we set $\eps_{pq}=\eps$ for all $p, q$.

\begin{defi}[Multispecies Boltzmann Entropy]
	Let $\f$ be a solution to the viscous multispecies BGK equation \eqref{eq:BGK_multi_viscous}. The multispecies entropy is given by the $H$-functional
	\begin{equation}
		\mathcal{H}(f) = \sum_{p=1}^M \int_{\R^d \times \R^d} f_p \log f_p \, dx dv.
	\end{equation}
\end{defi}

\begin{prop}[Multispecies Entropy Inequality]
\label{prop:H_Theorem_multi}
The solution $\f^\eps$ to the multispecies viscous BGK equation \eqref{eq:BGK_multi_viscous} where $\eps_{pq} = \eps$ for all $p,q$ satisfies the following entropy dissipation inequality
\begin{multline}
\label{eq:entropy_inequality_multi}
    \mathcal{H}(\f^\eps(t)) + \frac{1}{\eps} \int_0^t \mathcal{D}^{\textnormal{multi}}_{\textnormal{BGK}}(s) ds + 4\tau \sum_{p=1}^M \int_0^t \int_{\R^d \times \R^d} |\nabla_x \sqrt{f_p^\eps}|^2 \, dx dv ds \\ \leq \mathcal{H}(\f_0^\eps),
\end{multline}
where the relaxation dissipation is given by
\begin{multline}
    \mathcal{D}^{\textnormal{multi}}_{\textnormal{BGK}} := \sum_{p=1}^M \Biggl[ \int_{\R^d \times \R^d} (f_p^\eps - \mathcal{M}_{pp}[\f^\eps]) (1+ \log f_p^\eps) \, dx dv \\
    + \sum_{\substack{q=1\\q \neq p}}^M \int_{\R^d \times \R^d} (f_p^\eps - \mathcal{M}_{pq}[\f^\eps]) (1+ \log f_p^\eps) \, dx dv \Biggr]\geq 0.
\end{multline}
\end{prop}
\begin{proof}
The expression for the time derivative of the functional $\mathcal{H}$ can be obtained by multiplying the multispecies BGK equation by $(1+\log \f)$, where $\log \f = (\log f_p)_{p=1}^M$. The transport term vanishes due to the boundary conditions (periodic or decaying) as in the proof of Proposition \ref{prop:H_Theorem}. The intra-species collision operator is non-positive, since
\begin{equation*}
    \frac{1}{\eps_{pp}} \int (\mathcal{M}_{pp}[f_p^\eps] - f_p^\eps) (1 + \log f_p^\eps) \leq 0,
\end{equation*}
as for the single species case. The inter-species collision is more involved. Using the inequality
\begin{equation*}
    (y-x)(\log x+1) \leq y\log y - x\log x
\end{equation*}
we can estimate
\begin{equation*}
    \int (\mathcal{M}_{pq} - f_p) (\log f_p+1) \leq \int (\mathcal{M}_{pq} \log \mathcal{M}_{pq} - f_p \log f_p).
\end{equation*}
This expression can be further simplified by recalling the classical result in \cite{cercignani1988}, used also in \cite{bisiboscheri2022}, \textit{i.e.}
\begin{equation*}
    \int f_p \log f_p \geq \int \mathcal{M}_p \log \mathcal{M}_p,
\end{equation*}
which leads to
\begin{equation*}
    \int (\mathcal{M}_{pq} - f_p) (\log f_p+1) \leq \int (\mathcal{M}_{pq} \log \mathcal{M}_{pq} - \mathcal{M}_p \log \mathcal{M}_p) = - \frac{d}{2} n_p \log \left(\frac{T_{pq}}{T_q} \right).
\end{equation*}
Since it holds
\begin{equation*}
    \log(T_{pq}) - \log (T_p) \geq  \frac{n_q}{n_p + n_q} \log \left(\frac{T_{pq}}{T_q} \right),
\end{equation*}
we conclude that
\begin{equation*}
    \int (\mathcal{M}_{pq} - f_p) (\log f_p+1) + \int (\mathcal{M}_{qp} - f_q) (\log f_q+1) \leq 0.
\end{equation*}
The viscosity terms lead also in this case to a spatial Fisher-type term
\begin{equation*}
    \tau \sum_{p=1}^M \int \frac{|\nabla_x f_p^\eps|^2}{f_p^\eps}.
\end{equation*}
Finally, we can apply the identity $\displaystyle \frac{\vert{}\nabla_x f_p^\eps\vert{}^2}{f_p^\eps} = 4 \vert{}\nabla_x \sqrt{f_p^\eps}\vert{}^2$, to conclude. The addition of the spatial viscosity introduces a strictly positive dissipation mechanism which enforces the decay of the total entropy, leading to Equation \eqref{eq:entropy_inequality_multi}. 
\end{proof}

\subsection{Macroscopic Quantities for Segregated Species}
The macroscopic quantities defined in \eqref{eq:macro_multi}-\eqref{eq:total_macro_quantities_multi} are valid for situations in which the two species are not well separated. When we deal with segregated species, the  intra-species collisions become more frequent and the corresponding scaling must be properly chosen to correctly capture the physical behavior of particles. To this aim, following \cite{pupporeytenna2026}, we introduce the quantity $\alpha_p$, which describes the volume fraction occupied by the species $p$. Heuristically, to understand the role of this quantity, we can consider a control volume $V$ and define $\chi_p(x,t)$ as the probability of finding a particle of the species $p$ in $(x,t)$. The volume fraction $\alpha_p(x,t)$ occupied by the phase $p$ is then defined as
\begin{equation}
\label{volume_fraction}
    \alpha_p(x,t) = \frac{1}{|V|} \int_V \chi_p(x,t)\, dV,
\end{equation}
where $\sum_p \alpha_p = 1$.  The setting to keep in mind is sketched in Fig. \ref{Fig:interface_separation} for a mixture of two species.
To each point $x$ we associate a control volume $V$, which is large enough to contain several ``bubbles'', but smaller than the typical length scale.
In this setting,
\begin{align*}
	& \alpha_p \rho_p (x,t)= \frac{m_p}{\mu(V)}\int_V \int_{\R^d} \chi_p(x+y)\, f_p(x+y,v,t) d v\, d y, \\
	& \alpha_p\rho_p u_p(x,t) = \frac{m_p}{\mu(V)}\int_V \int_{\R^d} \chi_p(x+y) v f_p(x+y,v,t), \, dv\, d y.
\end{align*}
Further, we suppose that 
\begin{itemize}
	\item $\chi$ moves with the flow: $$\partial_t \chi_p + u_I \partial_x \chi_p = 0,$$
	where $u_I$ is the interface velocity.
	\item macroscopic quantities are almost constant within $V$, except for $\chi_p$.
\end{itemize}
	
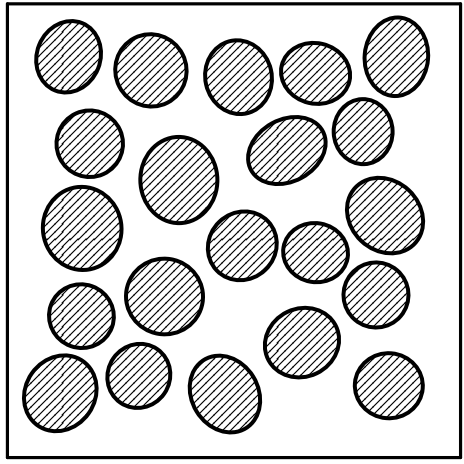
\begin{figure}
    \centering
\begin{tikzpicture}[line cap=round,line join=round,>=triangle 45,x=1cm,y=1cm]
	\fill[line width=1.5pt,color=white] (-3,3) -- (-3,-3) -- (3,-3) -- (3,3) -- cycle;
	\draw [line width=1.2pt] (-3,3)-- (-3,-3);
	\draw [line width=1.2pt] (-3,-3)-- (3,-3);
	\draw [line width=1.2pt] (3,-3)-- (3,3);
	\draw [line width=1.2pt] (3,3)-- (-3,3);
	
	% Border / Outer Region Particles
	\draw [rotate around={145.0:(-2.3,-2.15)},line width=1.5pt, fill=black,pattern=north east lines,pattern color=black] (-2.3,-2.15) ellipse (0.46cm and 0.52cm);
	\draw [rotate around={158.0:(-2.19,2.3)},line width=1.5pt, fill=black,pattern=north east lines,pattern color=black] (-2.19,2.3) ellipse (0.42cm and 0.48cm);
	\draw [rotate around={175.4:(2.05,-2.05)},line width=1.5pt, fill=black,pattern=north east lines,pattern color=black] (2.05,-2.05) ellipse (0.45cm and 0.43cm);
	\draw [rotate around={85.1:(2.15,2.3)},line width=1.5pt, fill=black,pattern=north east lines,pattern color=black] (2.15,2.3) ellipse (0.52cm and 0.42cm);
	\draw [rotate around={98.6:(-2.01,0.03)},line width=1.5pt, fill=black,pattern=north east lines,pattern color=black] (-2.01,0.03) ellipse (0.55cm and 0.52cm);
	\draw [rotate around={46.7:(2.0,0.2)},line width=1.5pt, fill=black,pattern=north east lines,pattern color=black] (2.0,0.2) ellipse (0.46cm and 0.54cm);
	\draw [rotate around={121.4:(-0.12,-2.16)},line width=1.5pt, fill=black,pattern=north east lines,pattern color=black] (-0.12,-2.16) ellipse (0.53cm and 0.44cm);
	\draw [rotate around={102.7:(0.06,2.03)},line width=1.5pt, fill=black,pattern=north east lines,pattern color=black] (0.06,2.03) ellipse (0.49cm and 0.44cm);
	\draw [rotate around={27.1:(-1.1,2.12)},line width=1.5pt, fill=black,pattern=north east lines,pattern color=black] (-1.1,2.12) ellipse (0.47cm and 0.48cm);
	\draw [rotate around={114.4:(-1.91,1.15)},line width=1.5pt, fill=black,pattern=north east lines,pattern color=black] (-1.91,1.15) ellipse (0.44cm and 0.44cm);
	\draw [rotate around={139.6:(1.88,-0.85)},line width=1.5pt, fill=black,pattern=north east lines,pattern color=black] (1.88,-0.85) ellipse (0.43cm and 0.43cm);
	\draw [rotate around={55.2:(-1.26,-1.92)},line width=1.5pt, fill=black,pattern=north east lines,pattern color=black] (-1.26,-1.92) ellipse (0.43cm and 0.41cm);
	\draw [rotate around={70.2:(-2.02,-1.13)},line width=1.5pt, fill=black,pattern=north east lines,pattern color=black] (-2.02,-1.13) ellipse (0.42cm and 0.43cm);
	\draw [rotate around={77.1:(1.08,2.08)},line width=1.5pt, fill=black,pattern=north east lines,pattern color=black] (1.08,2.08) ellipse (0.40cm and 0.46cm);
	
	% Interior / Central Particles
	\draw [rotate around={133.7:(0.11,-0.2)},line width=1.5pt, fill=black,pattern=north east lines,pattern color=black] (0.11,-0.2) ellipse (0.44cm and 0.47cm);
	\draw [rotate around={25.6:(0.9,-1.48)},line width=1.5pt, fill=black,pattern=north east lines,pattern color=black] (0.9,-1.48) ellipse (0.50cm and 0.45cm);
	\draw [rotate around={1.0:(-0.73,0.67)},line width=1.5pt, fill=black,pattern=north east lines,pattern color=black] (-0.73,0.67) ellipse (0.51cm and 0.57cm);
	\draw [rotate around={80.5:(-0.92,-0.87)},line width=1.5pt, fill=black,pattern=north east lines,pattern color=black] (-0.92,-0.87) ellipse (0.50cm and 0.51cm);
	\draw [rotate around={28.6:(0.7,1.06)},line width=1.5pt, fill=black,pattern=north east lines,pattern color=black] (0.7,1.06) ellipse (0.54cm and 0.41cm);
	\draw [rotate around={1.8:(1.71,1.31)},line width=1.5pt, fill=black,pattern=north east lines,pattern color=black] (1.71,1.31) ellipse (0.39cm and 0.43cm);
	\draw [rotate around={80.4:(1.08,-0.29)},line width=1.5pt, fill=black,pattern=north east lines,pattern color=black] (1.08,-0.29) ellipse (0.39cm and 0.43cm);
\end{tikzpicture}
	\caption{Diffuse interface gases in a control volume.}
	\label{Fig:interface_separation}
\end{figure}

\section{Formal Hydrodynamic Limit: Two-Phase Flow Model with two entropy inequalities}
In this section we discuss the formal hydrodynamic limit for the two-phase flow model of Romenski \textit{et al.} \cite{romenski2004, romenski2010} with two entropy inequalities from the multispecies viscous BGK model \eqref{eq:BGK_multi_viscous}, with $\eps_{pp}=\eps$ and $\eps_{pq}=1$ for $p \neq q$. For the sake of simplicity, we restrict to the spatial one-dimensional case and to the case of two species.\\ The governing equations for two-phase flows in the formulation given in \cite{romenski2007, thomann2023} consist of flux terms, source terms and artificial viscosity regularization terms (in blue). They read 
\begin{subnumcases}
	\displaystyle \partial_t(\alpha_1) + \bar{u}_1 \partial_x (\alpha_1) =  \textcolor{blue}{\tau \partial_{xx} \alpha_1}, \label{RDT_1}\\[10pt]
	\displaystyle \partial_t(\alpha_p \rho_p) + \partial_x (\alpha_p \bar{u}_p \rho_p) =  \textcolor{blue}{\tau \partial_{xx} (\alpha_p \rho_p)}, \qquad p=1,2 \label{RDT_2}\\[10pt]	
	\displaystyle \partial_t (\rho \bar{u} ) + \partial_x \Big(\alpha_1 \rho_1 \bar{u}_1^2 + \alpha_1 p_1 + \alpha_2 \rho_2 \bar{u}_2^2 + \alpha_2 p_2 \Big) = \textcolor{blue}{\tau \partial_{xx} (\rho \bar{u})}, \label{RDT_3}\\[10pt]
	\displaystyle \partial_t w + \partial_x 
	\left( \frac{u_1^2}{2}-\frac{u_2^2}{2} + \mu_1 - \mu_2 \right) = -\frac{\lambda_0}{\rho} + \textcolor{blue}{\tau \partial_{xx} (w)}, \label{RDT_4}
	\\[10pt]
	\displaystyle \partial_t (\alpha_p \rho_p s_p) + \partial_x 
	\left( \alpha_p \rho_p s_p \bar{u} \right) = \pi_p +  \textcolor{blue}{\tau \partial_{xx} (s)} + \textcolor{blue}{\Pi_p}, \qquad p=1,2 \label{RDT_5}
\end{subnumcases}
where $\alpha_p \in (0,1)$ is the volume fraction of phase $p$, $w=u_1-u_2$ is the relative velocity, $\alpha_p \rho_p s_p$ is the phase-entropy, $\tau>0$ is a given constant, while the partial pressures $p_p$, the chemical potentials $\mu_p$ and the source terms $\lambda_0$, $\Pi_p$ and $\pi_p$ will be defined later. In this framework, the total density $\rho$ is defined as 
\begin{equation}
    \rho = \alpha_1\,\rho_1 + \alpha_2\,\rho_2,
\end{equation}
and the mixture velocity $\bar{u}$ as
\begin{equation}
\label{mixture_velocity}
    \bar{u} = \kappa_1\,\bar{u}_1 + \kappa_2\,\bar{u}_2, \qquad \kappa_p = \frac{\alpha_p\,\rho_p}{\rho}.
\end{equation}

To derive the macroscopic equations without modelling the physical phenomena at interfaces, we adopt a diffuse interface approach at the kinetic level as done in \cite{pupporeytenna2026}. We assume the indicator function $\chi_p(t,x)$ is not a sharp step function, but rather a continuous phase-field that undergoes the same spatial regularization as the kinetic distribution. Therefore, it satisfies the advection-diffusion equation:
\begin{equation}
\label{eq:chi_diffuse}
    \partial_t \chi_p + u_I \partial_x \chi_p = \tau \partial_{xx} \chi_p,
\end{equation}
where $u_I$ is the characteristic velocity of the interface. Let us recall that $\chi_p$ jumps across the interfaces separating the two components, so its derivatives are actually delta functions centered on the interface, see \cite{drewpassman1998}.
% Thus, to close the system, \textcolor{red}{we assume that the kinetic distribution $f_p$ is locally homogeneous across the microscopic thickness of the diffuse interface. In other words, the indicator fraction jumps, but the distribution function does not, since molecular collisions across the microscopic interface ensure that the local fluid velocity and temperature remain continuous. This interfacial assumption implies that $\partial_x f_p \approx 0$ wherever $\partial_x \chi_p \neq 0$, allowing us to neglect the cross-derivative term}, namely
Thus, to close the system, we exploit the scale separation between the microscopic reference volume $V$ (representing the thickness of the diffuse interface) and the macroscopic reference length $L$ of the fluid flow. The kinetic distribution $f_p$ varies smoothly on the macroscopic scale $L$, meaning that its spatial gradients scale as $\partial_x f_p \sim O(1/L)$. Conversely, the indicator function $\chi_p$ varies abruptly over the microscopic interfacial scale. Since the interface thickness is negligible compared to the macroscopic scale $L$, the distribution function $f_p$ can be considered effectively constant across the interfacial region. This scale separation implies that $\partial_x f_p \approx 0$ at the local microscopic scale where $\partial_x \chi_p \neq 0$, allowing us to neglect the cross-derivative term, namely
\begin{equation}
\label{eq:assumption_interface_chi}
    \partial_x \chi_p \partial_x f_p \approx 0.
\end{equation}
In particular, it holds
\begin{align*}
	& \int_V \chi_p \int_{\R^d} \psi(v) \partial_t f_p = \int_V \int_{\R^d} \psi(v) \partial_t (f_p \chi_p)- \int_{V} \int_{\R^d} \psi(v) f_p \partial_t \chi_p 
	\\
	& \int_V \chi_p \int_{\R^d} \psi(v) \partial_x f_p = \int_V \int_{\R^d} \psi(v) \partial_x (f_p \chi_p)- \int_{V} \int_\R \psi(v) f_p \partial_x \chi_p \\
	& \int_V \chi_p \int_{\R^d} \psi(v) \partial_{xx} f_p = \int_V \int_{\R^d} \psi(v) \partial_{xx} (f_p \chi_p) - \int_{V} \int_{\R^d} \psi(v) f_p \partial_{xx} \chi_p.
\end{align*}

Equation \eqref{RDT_2} is obtained by considering $\psi(v)=1$ in the weak formulation of \eqref{eq:BGK_multi_viscous}, which yields the single-species mass conservation. Indeed, supposing macro quantities are approximately constant along the interface, we get
\begin{equation*}
    \partial_t (\alpha_p \rho_p) + \partial_x (\alpha_p \rho_p u_p)- \rho_{I} \int_V \left( \partial_t \chi_p + u_I \partial_x \chi_p - \partial_{xx} \chi_p \right)= \tau \partial_{xx} (\alpha_p \rho_p),
\end{equation*}
where, applying the transport equation \eqref{eq:chi_diffuse} for $\chi_p$, the integral term vanishes.\\

Let us now derive the first equation, which describes the evolution of the volume fraction, as already done in \cite{pupporeytenna2026}. Starting from \eqref{RDT_2}, we have
\begin{multline}
    0=\partial_t (\alpha_1 \rho_1) + \partial_x (\alpha_1 \bar{u}_1 \rho_1) - \tau \partial_{xx} (\alpha_1 \rho_1) \\= \rho_1 \partial_t \alpha_1 + \rho_1 \bar{u}_1 \partial_x \alpha_1 - \tau \partial_{xx} (\alpha_1) - 2 \partial_x \alpha_1 \partial_x \rho_1.
\end{multline}
Dividing the remaining equation by $\rho_1$, we retrieve the evolution equation for the volume fraction
\begin{equation}
    \partial_t \alpha_1 + \bar{u}_1 \partial_x \alpha_1 = \tau \partial_{xx} \alpha_1 + 2 \tau \frac{\partial_x \alpha_1 \partial_x \rho_1}{\rho_1}.
\end{equation}

The cross-derivative term vanishes, due to the assumption \eqref{eq:assumption_interface_chi} made on the structure of the diffuse interface. Since the volume fraction $\alpha_p$ is the macroscopic average of the indicator function $\chi_p$, we obtain
\begin{equation}
    \partial_x \alpha_1 \partial_x \rho_1 \approx 0.
\end{equation}
Applying this closure directly eliminates the cross-derivative term, yielding Equation \eqref{RDT_1}
\begin{equation*}
    \partial_t \alpha_1 + \bar{u}_1 \partial_x \alpha_1 = \tau \partial_{xx} \alpha_1.
\end{equation*}

Let us focus on the phase velocity equation \eqref{RDT_3} and the relative velocity equation \eqref{RDT_4}. Expand the single-species momentum equation as
\begin{equation}
    \rho_p \partial_t \bar{u}_p + \bar{u}_p \left[ \partial_t \rho_p + \partial_x ( \rho_p \bar{u}_p) \right] + \rho_p \bar{u}_p \partial_x \bar{u}_p + \partial_x p_p = R_p,
\end{equation}
where the term in the brackets is zero due to single-species mass conservation and $R_p = \xi_{pq} (u_p - u_q)$ represents the momentum exchange with a constant $\xi_{pq}=\xi_{qp}$ which can be explicitly computed, see  \cite[Appendix A]{pupporeytenna2026}. Summing over the species $p$ we easily retrieve Equation \eqref{RDT_3}. On the other hand, dividing the entire equation by $\rho_p$ yields the non-conservative velocity equation
\begin{equation}
    \partial_t \bar{u}_p + \bar{u}_p \partial_x \bar{u}_p + \frac{1}{\rho_p} \partial_x p_p = \frac{R_p}{\rho_p},
\end{equation}
which can be rewritten as
\begin{equation}
    \partial_t \bar{u}_p + \partial_x \left(\frac{1}{2}\bar{u}_p^2 \right) + \frac{1}{\rho_p} \partial_x p_p = \frac{R_p}{\rho_p}.
\end{equation}
To keep consistency with the model introduced in \cite{romenski2007, thomann2023}, let us define the chemical potential (Gibbs free energy) as
\begin{equation}
    \mu_p := e_p + \frac{p_p}{\rho_p} - s_p T_p .
\end{equation}
Let us investigate the explicit expression of the spatial derivative of this chemical, recalling that the internal energy satisfies $de_p = T_p d_p + \frac{p_p}{\rho_p^2} d\rho_p$, 
\begin{equation}
\label{eq:derivative_gibbs}
    \partial_x \mu_p = \partial_x e_p +\frac{1}{\rho_p} \partial_x p_p - \frac{p_p}{\rho_p^2} \partial_x \rho_p - T_p \partial_x s_p - s_p \partial_x T_p = \frac{1}{\rho_p} \partial_x p_p - s_p \partial_x T_p.
\end{equation}
Then, subtracting the phase-$2$ velocity equation from the phase-$1$ velocity equation yields
\begin{equation}
    \partial_t (\bar{u}_1 - \bar{u}_2) + \partial_x \left( \frac{1}{2} \bar{u}_1^2 - \frac{1}{2} \bar{u}_2^2 \right) + \left( \frac{\partial_x p_1}{\rho_1} - \frac{\partial_x p_2}{\rho_2} \right) = \frac{R_{1,2}}{\rho_1} - \frac{R_{2,1}}{\rho_2}
\end{equation}

Substituting the expression derived in \eqref{eq:derivative_gibbs}, we get
\begin{equation}
    \partial_t w + \partial_x \left( \frac{1}{2} \bar{u}_1^2 - \frac{1}{2} \bar{u}_2^2 \right) + \partial_x \left( \mu_1 - \mu_2 \right) = - \left(s_1 \partial_x T_1 - s_2 \partial_x T_2 \right) + \frac{R_{1,2}}{\rho_1} - \frac{R_{2,1}}{\rho_2},
\end{equation}
where the RHS is equivalent to the source term $-\lambda_0/\rho$ defined in 
\cite[Section 4.4]{romenski2007}.

\begin{rem}
The relative velocity equation \eqref{RDT_4} can be also rewritten following \cite[Equation (3e)]{thomann2023} as 
\begin{equation*}
\displaystyle \partial_t w + \partial_x 
	\left( \bar{u} w + E_c \right) + \bar{u} \partial_x w= -\frac{\lambda_0}{\rho} + \tau \partial_{xx} (w),
\end{equation*}
where $E_c$ is defined as
\begin{equation*}
    E_c := \mu_1 - \mu_2 + (1-2c_1) \frac{w^2}{2}.
\end{equation*}
Indeed, using the expression of the average velocity mixture and recalling that $c_1 + c_2 = 1$, we observe that
\begin{equation*}
    \bar{u}_1 = \bar{u} + c_2 w, \qquad \bar{u}_2 = \bar{u} - c_1 w,
\end{equation*}
which implies
\begin{equation*}
    \bar{u}_1^2 - \bar{u}_2^2 = 2\bar{u}(c_2 w + c_1 w) + (c_2^2-c_1^2) w^2 = 2 \bar{u} w + (1-2c_1)w^2.
\end{equation*}
Substituting this expanded term back in the equation for the relative velocity and grouping the terms under the spatial derivative
\begin{equation}
\label{eq:relative_velocity}
     \partial_t w + \partial_x \left( \bar{u} w + \left[\mu_1 - \mu_2 + (1-2c_1) \frac{w^2}{2}\right] \right) = - \left(s_1 \partial_x T_1 - s_2 \partial_x T_2 \right) + \frac{R_{1,2}}{\rho_1} - \frac{R_{2,1}}{\rho_2},
\end{equation}
where the term in square brackets is exactly $E_c$. Let us study the consistency of the source term with the one appearing in \cite[Equation (3e)]{thomann2023}. The relative velocity relaxation source is given by
\begin{equation}
\label{eq:lambda0}
    \lambda^0 = \chi^0 c_1 c_2 w + \chi_1 c_1 A^1 j_1 + \chi_2 c_2 A^2 j_2,
\end{equation}
where
\begin{equation*}
    \chi_1 = - \frac{\rho s_1}{c_1 \kappa^1}, \qquad \chi_2 = + \frac{\rho s_2}{c_2 \kappa^2},
\end{equation*}
and $j_1$, $j_2$ are independent variables called \textit{thermal impulses}. Substituting them, we get
\begin{equation*}
    \lambda^0 = \chi_0 c_1 c_2 w - \rho s_1 \left( \frac{A^1}{\kappa^1} j_1 \right) + \rho s_2 \left( \frac{A^2}{\kappa^2} j_2 \right)
\end{equation*}
Rather than treating these thermal impulses as independent dynamical variables, we rely on their asymptotic behavior described in \cite[Section 5.2]{Dumbser2017}. Since in the asymptotic limit of the heat flux it holds 
\begin{equation*}
    \partial_x T_p \approx - \frac{A^1}{\kappa^1} j^1,
\end{equation*}
we can substitute $-\partial_x T_1$ and $-\partial_x T_2$ back into \eqref{eq:lambda0}. This gives 
\begin{equation}
    \lambda^0 \approx \chi^0 c_1 c_2 w - \rho s_1 (-\partial_x T_1) + \rho s_2 (-\partial_x T_2)
\end{equation}
and, dividing our result by $-\rho$, yields
\begin{equation}
    -\frac{\lambda^0}{\rho} \approx - \frac{\chi^0 c_1 c_2}{\rho} w - (s_1 \partial_x T_1 - s_2 \partial_x T_2),
\end{equation}
which is the RHS of \eqref{eq:relative_velocity}.
\end{rem}
Let us finally focus on the entropy inequality \eqref{RDT_5}. The derivation of the entropy flux is completely analogous to the single-species case presented in Section \ref{Section:entropy_single}. For this reason, we mainly focus on the source terms: the term $\tau \partial_{xx}(s_p) + \Pi_p$ comes from the viscosity regularization term, which is equivalent to the single-species case, whereas the term $\pi_p$ comes from the inter-species collision term.\\
The terms coming from the viscosity regularization terms can be derived following the same strategy introduced for the single-species case in Section \ref{Section:entropy_single}, which leads to 
\begin{equation}
    \Pi_p = \frac{1}{T_p} \partial^2_q \mathcal{E}_p.
\end{equation}
Let us explicitly compute the inter-species entropy source term for the species $2$. 
\begin{multline}
    \int (\mathcal{M}_{21}-\mathcal{M}_2) \log \mathcal{M}_2 \,dv = \\
     \int \left[ \frac{n_2 m_2^{d/2}}{(2 \pi T_{21})^{d/2}} \exp \left(-\frac{m_2 (\bar{u}_{21}-v)^2}{2T_{21}}\right) - \frac{n_2 m_2^{d/2}}{(2 \pi T_2)^{d/2}} \exp \left(-\frac{m_2 (\bar{u}_{2}-v)^2}{2T_{2}}\right) \right] \\ \times \left[ \log \left( \frac{n_2 m_2^{d/2}}{(2 \pi T_2)^{d/2}}\right)  - \frac{m_2 (\bar{u}_{2}-v)^2}{2T_{2}} \right] \,dv \\ 
     = \frac{m_2}{2T_2} \left[ \left( n_2 \frac{d T_{21}}{m_2} + n_2 (\bar{u}_{21} - \bar{u}_2)^2 \right) - n_2 \frac{d T_2}{m_2} \right] \\
     = n_2 \frac{d}{2} \left( \frac{T_{21}}{T_2} - 1 \right) + n_2 \frac{m_2}{2T_2} (\bar{u}_{21}-\bar{u}_2)^2.
\end{multline}
Concerning the inter-species entropy source term for the species $1$, the expression is analogous with $\cdot_2 \leftrightarrow \cdot_1$ and $\bar{u}_{21} \leftrightarrow \bar{u}_{12}$.
Let us reassemble the expression to retrieve the source term of \cite{thomann2023}. Let us observe that, by conservation of momentum, we have
\begin{equation}
    m_1 \nu_{12} n_1 (\bar{u}_{12}-u_1) + m_2 \nu_{21} n_2 (\bar{u}_{21}-u_2) = 0. 
\end{equation}
Let us assume that $\bar{u}_{21}=\bar{u}_{12}=\bar{u}_{mix}$ and $\rho_1 = m_1 n_1$, $\rho_2 = m_2 n_2$, then
\begin{equation*}
    \rho_1 \nu_{12} (\bar{u}_{mix} - \bar{u}_1)  + \rho_2 \nu_{21} (\bar{u}_{mix} - \bar{u}_2) = 0, 
\end{equation*}
which implies
\begin{equation*}
    \bar{u}_{mix} = \frac{\rho_1 \nu_{12} \bar{u}_1 + \rho_2 \nu_{21} \bar{u}_2}{\rho_1 \nu_{12} + \rho_2 \nu_{21}}.
\end{equation*}
Now, 
\begin{equation*}
    \bar{u}_{21} - \bar{u}_2 = \bar{u}_{mix} - \bar{u}_2 = \frac{\nu_{12} \rho_1}{\rho_1 \nu_{12} + \rho_2 \nu_{21}} (\nu_{12} \bar{u}_1 - \nu_{21} \bar{u}_2)
\end{equation*}
Then, the source term for the entropy production due to the inter-species terms is given by
\begin{equation*}
    n_2 \frac{d}{2} \left(\frac{T_{21}}{T_2}-1 \right) + \frac{\nu_{12} \rho_1 \rho_2}{(\nu_{12} \rho_1+\nu_{21} \rho_2)^2}\left( \frac{\rho_2}{2T_2} \right) w^2,
\end{equation*}
where we observe that $\frac{\rho_1 \rho_2}{(\rho_1+\rho_2)^2} = c_1 c_2$, setting $\nu_{12} = \alpha_1$ and $\nu_{21} = \alpha_2$.\\
According to Equation \eqref{eq:temperature_mixture}, we have
\begin{equation}
    T_{12} = T_{21} =  \underbrace{\frac{n_1 T_1 + n_2 T_2}{n_1 + n_2}}_{\bar{T}} + \kappa w^2.
\end{equation}
Substituting, we obtain
\begin{equation}
    \pi_p = \frac{d}{2} \left[n_2 \left( \frac{\bar{T} + \kappa w^2}{T_2} - 1 \right) \right] + c_1 c_2 \left( \frac{\rho_2}{2T_2} \right) w^2
\end{equation}

Let us collect all the terms containing the relative velocity term, namely
\begin{equation}
    \pi_{2, w} = \frac{w^2}{T_2} \left[ \frac{d}{2} n_2 \kappa + \frac{c_1 c_2 \rho_2}{2} \right] \ge 0.
\end{equation}
Now let us extract the terms containing $\bar{T}$
\begin{equation}
    \pi_{2, \bar{T}} = \frac{d}{2} n_2 \left( \frac{\bar{T}}{T_2} - 1 \right),
\end{equation}
which, using the definition $\bar{T}$, are simplified to 
\begin{equation}
    \pi_{2, \bar{T}} = \frac{d}{2} n_2 \left( \frac{n_1 T_1 - n_1 T_2}{T_2(n_1 + n_2)} \right),
\end{equation}
or, equivalently, to
\begin{equation}
    \pi_{2, \bar{T}} = \frac{d}{2} \left( \frac{n_1 n_2}{n_1 + n_2} \right) \frac{T_1 - T_2}{T_2}.
\end{equation}

Combining the different contributions, the kinetic entropy production for species 2 is
\begin{equation}
    \pi_2 = \frac{d}{2} \left( \frac{n_1 n_2}{n_1 + n_2} \right) \frac{T_1 - T_2}{T_2} + \frac{w^2}{T_2} \left[ \frac{d}{2} n_2 \kappa + \frac{c_1 c_2 \rho_2}{2} \right].
\end{equation}

The single-species entropy, $\pi_p$, is not required to be positive, since the total entropy of the closed system is given by the sum over all species. Therefore, the quantity that must be non-negative is $\sum_p \pi_p$. Let us retrieve the expression for the total mixture entropy production, which is given by
\begin{multline}
\label{eq:entropy_prod_total_mixture}
    \pi_1+\pi_2 = \frac{d}{2} \left( \frac{n_1 n_2}{n_1 + n_2} \right) \left[ \frac{T_2 - T_1}{T_1} + \frac{T_1 - T_2}{T_2} \right] \\+ w^2 \left[ \frac{1}{T_1} \left( \frac{d}{2} n_1 \kappa + \frac{c_1 c_2 \rho_1}{2} \right) + \frac{1}{T_2} \left( \frac{d}{2} n_2 \kappa + \frac{c_1 c_2 \rho_2}{2} \right) \right].
\end{multline}

Let us first factor out $(T_1 - T_2)$ to get
\begin{equation*}
    \left[ \frac{-(T_1 - T_2)}{T_1} + \frac{T_1 - T_2}{T_2} \right] = (T_1 - T_2) \left( \frac{1}{T_2} - \frac{1}{T_1} \right) = \frac{(T_1 - T_2)^2}{T_1 T_2},
\end{equation*}
from which we have 
\begin{equation*}
    \left[ \frac{d \kappa}{2} \left( \frac{n_1}{T_1} + \frac{n_2}{T_2} \right) + \frac{c_1 c_2}{2} \left( \frac{\rho_1}{T_1} + \frac{\rho_2}{T_2} \right) \right]
\end{equation*}
Therefore, substituting back into \eqref{eq:entropy_prod_total_mixture} and defining $\Delta T = T_1 - T_2$, we get the final compact expression
\begin{multline}
    \pi_1+\pi_2 = \underbrace{\left[ \frac{d}{2} \left( \frac{n_1 n_2}{n_1 + n_2} \right) \frac{1}{T_1 T_2} \right]}_{C_{\Delta T} > 0} (\Delta T)^2 \\+ \underbrace{\left[ \frac{d \kappa}{2} \left( \frac{n_1}{T_1} + \frac{n_2}{T_2} \right) + \frac{c_1 c_2}{2} \left( \frac{\rho_1}{T_1} + \frac{\rho_2}{T_2} \right) \right]}_{C_w > 0} w^2
\end{multline}

The derivation explicitly proves that the total entropy production is a linear combination of the squares of the thermodynamic driving forces $\pi_1+\pi_2 \propto (\Delta T)^2 + w^2$, coherently with \cite{romenski2004, romenski2007}. Since $C_{\Delta T}$ and $C_w$ are strictly positive, the total mixture entropy is positive and its Hessian matrix is positive definite, thus convex.

\section{Conclusions}
In this paper, we have investigated the viscous BGK model of the Boltzmann equation and its hydrodynamic limit. From a rigorous point of view, we have proved convergence to the viscous Euler equations, a widely studied system in the context of vanishing viscosity regularization to justify the selection of physically admissible shock waves. In this perspective, the spatial Fisher information bound provides essential control over the kinetic entropy, delivering the structural regularity necessary to derive rigorous a priori estimates and validate the hydrodynamic limit.\\
From a formal point of view, we have established a bridge between kinetic theory and the SHTC framework. In particular, we retrieve the Euler equations with an entropy inequality, where the energy conservation is an extra conservation law. Differently from the SHTC framework, the entropy production term derived from the BGK equation automatically ensures the energy conservation at the macroscopic level. The analysis has been extended for the multispecies BGK model, retrieving in the limit the two-phase flow system with two entropy inequalities proposed in \cite{romenski2010}. Also in this case, the entropy production term can be consistently recovered from the kinetic formulation, without further assumptions.

\section*{Acknowledgement}
The authors received funding from the European Union's Horizon Europe research and innovation program under the Marie Skłodowska-Curie Doctoral Network DataHyking (Grant No. 101072546). GP and TT are members of the INdAM Research National Group of Scientific Computing (INdAM-GNCS).\\

\subsection*{Data Availability} Data sharing not applicable to this article as no datasets were generated or analysed during the current study.

\subsection*{Conflict of interest} No conflict of interest is extant in the present work.

\appendix

\bibliography{References}
\bibliographystyle{acm}

\end{document}